\documentclass[10pt,twoside,reqno]{amsart}
\usepackage{amssymb}
\usepackage{multirow}
\calclayout

\numberwithin{equation}{section}
\makeatletter

\renewcommand{\@secnumfont}{\bfseries}

\renewcommand{\section}{\@startsection{section}{1}%
  {0mm}{.7\linespacing\@plus\linespacing}{.5\linespacing}
  {\normalfont\bfseries\centering}}

\newcommand{\bibsection}{\@startsection{section}{1}%
  {0mm}{.7\linespacing\@plus\linespacing}{.5\linespacing}
  {\normalfont\scshape\centering}}

\renewcommand{\@biblabel}[1]{#1.}

\newtheorem{thm}{\bf Theorem}[section]

\newtheorem{cor}[thm]{\bf Corollary}

\begin{document}

\vspace{1.3cm}

\title{Some identities for degenerate Bernoulli numbers of the second kind}

\author{Taekyun Kim}
\address{Department of Mathematics, Kwangwoon University, Seoul 139-701, Republic of Korea}
\email{tkkim@kw.ac.kr}

\author{Dae San Kim}
\address{Department of Mathematics, Sogang University, Seoul 121-742, Republic of Korea}
\email{dskim@sogang.ac.kr}

\begin{abstract}
We introduce the degenerate Bernoulli numbers of the  second kind
as a degenerate version of the Bernoulli numbers of the second
kind. We derive a family of nonlinear differential equations
satisfied by a function closely related to the generating function
for those numbers. We obtain explicit expressions for the
coefficients appearing in those differential equations and the
degenerate Bernoulli numbers of the second kind. In addition, as
an application and from those differential equations we have an
identity expressing the degenerate Bernoulli numbers of the second
kind in terms of those numbers of higher-orders.
\end{abstract}
\maketitle
\bigskip
\medskip

\markboth{\centerline{\scriptsize Some identities for degenerate Bernoulli numbers of the second kind }}
{\centerline{\scriptsize T. Kim, D. S. Kim}}

\section{Introduction and preliminaries}

As is well-known, the Bernoulli polynomials $B_n(x)$ are defined by
\begin{equation}\begin{split}\label{101}
\frac{t}{e^t-1}e^{xt} = \sum_{n=0}^\infty    B_n(x)   \frac{t^n}{n!}.
\end{split}\end{equation}

A degenerate version of the Bernoulli polynomials $B_n(x)$, denoted by $\beta_n(x)$ and called the degenerate Bernoulli polynomials, was introduced in [1,2] by Carlitz. They are given by
\begin{equation}\begin{split}\label{102}
\frac{t}{(1+\lambda t)^{\frac{1}{\lambda }}-1}(1+\lambda t)^{\frac{x}{\lambda }} = \sum_{n=0}^\infty      \beta_n(x) \frac{t^n}{n!}.
\end{split}\end{equation}

When $x=0$, $\beta_n=\beta_n(0)$ are called the degenerate Bernoulli numbers.

In \cite{09}, the degenerate exponential function $e_\lambda (t)$, $(\lambda \in (0,\infty), t \in \mathbb{R})$ was introduced in order to study the degenerate gamma function $\Gamma_\lambda (s)$ $(0< Re(s)<\frac{1}{\lambda})$, and the degenerate Laplace transforms $\mathcal{L}_\lambda  (f(t))$ for functions $f(t)$ defined for $t \geq 0$. The degenerate exponential function $e_\lambda (t) $ is defined by
\begin{equation}\begin{split}\label{103}
e_\lambda (t) = (1+\lambda t)^{\frac{1}{\lambda }}.
\end{split}\end{equation}

Thus, in terms of $e_\lambda (t)$, the degenerate Bernoulli numbers $\beta_n$ are given by
\begin{equation}\begin{split}\label{104}
\frac{t}{e_\lambda (t)-1} = \sum_{n=0}^\infty  \beta_n \frac{t^n}{n!}.
\end{split}\end{equation}

The compositional inverse of $e_\lambda (t)$, denoted by $\log_\lambda t$ and called the degenerate logarithmic function, is given by
\begin{equation}\begin{split}\label{105}
\log_\lambda  t = \frac{t^\lambda -1}{\lambda }.
\end{split}\end{equation}

The Bernoulli polynomials of the second kind $b_n(x)$ are defined by
\begin{equation}\begin{split}\label{106}
\frac{t}{\log(1+t)} (1+t)^x = \sum_{n=0}^\infty     b_n(x)   \frac{t^n}{n!}.
\end{split}\end{equation}

For $x=0$, $b_n=b_n(0)$ are called the Bernoulli numbers of the second kind.

In view of these considerations, the degenerate Bernoulli numbers of the second kind $b_{n,\lambda }$ are naturally introduced as
\begin{equation}\begin{split}\label{107}
\frac{t}{\log_\lambda (1+t)} = \frac{\lambda t}{(1+t)^\lambda -1} = \sum_{n=0}^\infty        b_{n,\lambda }\frac{t^n}{n!},\,\,(\lambda  \neq 0).
\end{split}\end{equation}

In this paper, we consider the function
\begin{equation}\begin{split}\label{108}
F=F(t)=F(t;\lambda ) = \frac{1}{\log_\lambda (1+t)} = \frac{\lambda }{(1+t)^\lambda -1},
\end{split}\end{equation}

\noindent which is closely related to the generating function for
$b_{n,\lambda }$. We will derive a family of nonlinear differential
equations satisfied by $F(t)$, and apply these to find some identity
expressing $b_{n,\lambda }$ in terms of higher-order degenerate
Bernoulli numbers of the second kind $b_{n,\lambda }^{(r)}$ (see
\eqref{50}). This line of study has been very active in recent
years, some of which are [4-7,11-13, 17].

Those family of differential equations involve certain coefficients $a_{i,\lambda }(N)$, which are given by and uniquely determined  by recurrence relations. One particular thing to note here is that those coefficients are explicitly determined in terms of the degenerate Stirling numbers of the second kind $S_{2,\lambda }(n,k)$, and also of the falling factorial polynomials.

The explicit formulas were obtained by using Fa\`a di Bruno formula (\cite{03}, p.137) which is expressed in terms of the exponential partial Bell polynomials $B_{n,k}(x_1,x_2,\cdots,x_{n-k+1})$ (see \eqref{18}). This idea of adopting Fa\`a di Bruno formula has been fruitfully exploited by F. Qi and their colleagues. For this, we let the reader refer to \cite{14,15} and the references therein.

The degenerate Stirling numbers of the first kind $S_{1,\lambda }(n,k)$ and those of the second kind $S_{2,\lambda }(n,k)$ had been recently introduced respectively in \cite{13} and \cite{08,13}. Further, the related degenerate complete Bell and degenerate exponential partial Bell polynomials were also introduced in \cite{07}.

For each nonnegative integer $k$, the degenerate Stirling numbers of the second kind $S_{2,\lambda }(n,k)$ are given by
\begin{equation}\begin{split}\label{109}
\frac{1}{k!} ((1+\lambda t)^{\frac{1}{\lambda }}-1)^k = \sum_{n=k}^\infty S_{2,\lambda }(n,k) \frac{t^n}{n!}.
\end{split}\end{equation}

Also, the generalized falling factorial polynomials $(x)_{n,\lambda }$ are defined by
\begin{equation}\begin{split}\label{110}
(x)_{n,\lambda }=x(x-\lambda )\cdots(x-(n-1)\lambda ),\,\,(n \geq 1),\,\,(x)_{0,\lambda }=1.
\end{split}\end{equation}

In particular, the falling factorial polynomials $(x)_n$ are given by $(x)_n=(x)_{n,1}$, $(n \geq 0)$.

Here we state in the following two theorems which are snapshots of this paper.

\begin{thm}
The family of nonlinear differential equations
\begin{equation}\begin{split}\label{111}
(-1)^N (1+t)^N F^{(N)} = \sum_{i=0}^N a_{i,\lambda }(N) F^{i+1},\,\,(N=1,2,3,\cdots)
\end{split}\end{equation}
have a solution
\begin{equation*}\begin{split}
F=F(t;\lambda )=\frac{1}{\log_\lambda (1+t)}=\frac{\lambda }{(1+t)^\lambda -1},
\end{split}\end{equation*}
where, for $0 \leq i \leq N$.
\begin{equation}\begin{split}\label{112}
a_{i,\lambda }(N) &= (-1)^N \lambda ^{-i} \sum_{k=i}^N \sum_{l=0}^k (-1)^l {k \choose i} {k \choose l} (\lambda l)_N\\
&= (-1)^N \lambda ^{N-i} \sum_{k=i}^N (-1)^k k! {k \choose i} S_{2,\frac{1}{\lambda }}(N,k).
\end{split}\end{equation}
\end{thm}

\begin{thm} The degenerate Bernoulli numbers of the second kind $b_{n,\lambda }$ are given by
\begin{equation}\begin{split}\label{113}
b_{n,\lambda } &= (-1)^n \bigg\{ \frac{(1)_{n+1,\lambda }}{n+1} + \sum_{i=0}^{n-1} \frac{(1)_{i+1,\lambda }}{(i+1)!} (a_{i,\lambda }(n)-na_{i,\lambda }(n-1))\bigg\}\\
&=\frac{(-1)^n (1)_{n+1,\lambda }}{n+1} + \sum_{i=0}^{n-1} \sum_{k=i}^n \frac{(1)_{i+1,\lambda }}{(i+1)!} (-1)^k k! {k \choose i} \lambda ^{n-i-1} \\
&\quad \times (\lambda S_{2,\frac{1}{\lambda }}(n,k) + nS_{2,\frac{1}{\lambda }}(n-1,k))\\
&=\frac{(-1)^n (1)_{n+1,\lambda }}{n+1} + \sum_{i=0}^{n-1} \frac{(1)_{i+1,\lambda }}{(i+1)!}\lambda ^{-i}\\
&\quad \times \bigg\{ \sum_{l=0}^n (-1)^l {n \choose i} {n \choose l} (\lambda l)_n + \sum_{k=i}^{n-1} \sum_{l=0}^k (-1)^l {k \choose i} {k \choose l} (\lambda l+1)_n \bigg\},
\end{split}\end{equation}
where we understand that $S_{2,\frac{1}{\lambda }}(n-1,n)=0$.
\end{thm}

In summary, we introduce the degenerate Bernoulli numbers of the second kind $b_{n,\lambda }$ as a degenerate version of Bernoulli numbers of the second kind. We derive a family of nonlinear differential equations in \eqref{111} which are satisfied by the function $F(t)$ closely related to the generating function $tF(t)$ of the numbers $b_{n,\lambda }$. We obtain the explicit expressions in \eqref{112} for the coefficients $a_{i,\lambda }(N)$ appearing in those differential equations by using Fa\`a di Bruno formula (see \eqref{21}). Then from the differential equations in \eqref{111} and the explicit expressions of $a_{i,\lambda }(N)$ in \eqref{112} we will determine the explicit expressions of $b_{n,\lambda }$ in \eqref{113}. Finally, as an application we will derive an identity which expresses $b_{n,\lambda }$ in terms of higher-order degenerate Bernoulli numbers of the second.

\section{Differential equations satisfied by degenerate Bernoulli numbers of the second kind}

We let
\begin{equation}\begin{split}\label{01}
F=F(t) = F(t;\lambda ) = \frac{1}{\log_\lambda (1+t)},\,\,(\lambda  \neq 0).
\end{split}\end{equation}

Differentiation with respect to $t$ of \eqref{01} gives us
\begin{equation}\begin{split}\label{02}
F^{(1)} &= \frac{-1}{(\log_\lambda (1+t))^2}(1+t)^{\lambda -1}\\
&= \frac{-1}{(\log_\lambda (1+t))^2} \left( \frac{\lambda }{1+t}\log_\lambda (1+t) + \frac{1}{1+t} \right)\\
&= \frac{-1}{1+t} \left( \frac{\lambda }{\log_\lambda (1+t)} + \left( \frac{1}{\log_\lambda (1+t)} \right)^2 \right)\\
&= \frac{-1}{1+t} (\lambda F+F^2 ).
\end{split}\end{equation}

Further differentiation of \eqref{02} yields
\begin{equation}\begin{split}\label{03}
F^{(2)} &= \left( \frac{-1}{1+t} \right)^2 (\lambda F+F^2 ) + \frac{-1}{1+t}(\lambda F^{(1)} + 2FF^{(1)} ) \\
&= \left( \frac{-1}{1+t} \right)^2 (\lambda F+F^2)(1+\lambda +2F) \\
&=\left( \frac{-1}{1+t} \right)^2 \left\{ (\lambda +\lambda ^2)F + (1+3\lambda )F^2 + 2F^3 \right\}.
\end{split}\end{equation}

Yet further differentiation gives us
\begin{equation}\begin{split}\label{04}
F^{(3)} &= \frac{-2}{(1+t)^3}\left\{ (\lambda +\lambda ^2)F + (1+3\lambda )F^2 + 2F^3 \right\}\\
&\quad + \left( \frac{-1}{1+t} \right)^2 F^{(1)} \left\{ (\lambda +\lambda ^2)+ (2+6\lambda )F + 6F^2 \right\}\\
&= \left( \frac{-1}{1+t} \right)^3 \left\{ (2\lambda +2\lambda ^2)F + (2+6\lambda )F^2 + 4F^3 \right\}\\
&\quad + \left( \frac{-1}{1+t} \right)^3 (\lambda F+F^2) \left\{ (\lambda +\lambda ^2)+ (2+6\lambda )F + 6F^2 \right\}\\
&= \left( \frac{-1}{1+t} \right)^3 \left\{ (2\lambda +3\lambda ^2+\lambda ^3)F + (2+9\lambda +7\lambda ^2)F^2 + (6+12\lambda )F^3 + 6F^4 \right\}.
\end{split}\end{equation}

The above observations lead us to put
\begin{equation}\begin{split}\label{05}
F^{(N)} = \left( \frac{-1}{1+t} \right)^N \sum_{i=0}^N a_{i,\lambda }(N) F^{i+1},\,\,(N=1,2,3,\cdots).
\end{split}\end{equation}

Our next task is to determine $a_{i,\lambda }(N)$, $(0 \leq i \leq N)$.

By differentiating \eqref{05}, we obtain
\begin{equation*}\begin{split}\label{06}
F^{(N+1)} &= N \left( \frac{-1}{1+t} \right)^{N+1} \sum_{i=0}^N a_{i,\lambda }(N) F^{i+1}\\
&\quad +  \left( \frac{-1}{1+t} \right)^N F^{(1)} \sum_{i=0}^N (i+1) a_{i,\lambda }(N) F^i\\
&=  \left( \frac{-1}{1+t} \right)^{N+1} \sum_{i=0}^N N a_{i,\lambda }(N) F^{i+1}\\
&\quad +  \left( \frac{-1}{1+t} \right)^{N+1}  \sum_{i=0}^N (i+1) a_{i,\lambda }(N) (\lambda F+F^2) F^i\\
&=  \left( \frac{-1}{1+t} \right)^{N+1} \sum_{i=0}^N (N+\lambda +i\lambda ) a_{i,\lambda }(N) F^{i+1}\\
&\quad +  \left( \frac{-1}{1+t} \right)^{N+1}  \sum_{i=0}^N (i+1) a_{i,\lambda }(N) F^{i+2}
\end{split}\end{equation*}
\begin{equation}\begin{split}\label{06}
&=  \left( \frac{-1}{1+t} \right)^{N+1} \sum_{i=0}^N (N+\lambda +i\lambda ) a_{i,\lambda }(N) F^{i+1}\\
&\quad +  \left( \frac{-1}{1+t} \right)^{N+1}  \sum_{i=1}^{N+1} i a_{i-1,\lambda }(N) F^{i+1}
\\
&=  \left( \frac{-1}{1+t} \right)^{N+1} \bigg\{ (N+\lambda )a_{0,\lambda }(N)F \\
&\quad + \sum_{i=1}^N \big((N+\lambda +i\lambda )a_{i,\lambda }(N)+ia_{i-1,\lambda }(N) \big) F^{i+1} \\
&\quad + (N+1)a_{N,\lambda }(N)F^{i+2} \bigg\}.
\end{split}\end{equation}

On the other hand, by replacing $N$ by $N+1$ in \eqref{05}, we have
\begin{equation}\begin{split}\label{07}
F^{(N+1)} = \left( \frac{-1}{1+t} \right)^{N+1} \sum_{i=0}^{N+1} a_{i,\lambda }(N+1) F^{i+1}.
\end{split}\end{equation}

Comparing \eqref{06} and \eqref{07}, we get the following recurrence relations:
\begin{eqnarray}\begin{split}\label{08}
&a_{0,\lambda }(N+1)=(N+\lambda )a_{0,\lambda }(N),\\
&a_{N+1,\lambda }(N+1) = (N+1)a_{N,\lambda }(N),\\
&a_{i,\lambda }(N+1) = (N+(i+1)\lambda )a_{i,\lambda }(N)+ia_{i-1,\lambda }(N), \,\,\text{for}\,\, 1 \leq i \leq N.
\end{split}\end{eqnarray}

From \eqref{05} with $N=1$ and \eqref{02}, we obtain the initial conditions
\begin{equation}\begin{split}\label{11}
a_{0,\lambda }(1) = \lambda, \,a_{1,\lambda }(1)=1.
\end{split}\end{equation}

We now observe from \eqref{08} and \eqref{11} that
\begin{equation}\begin{split}\label{12}
a_{0,\lambda }(N+1)&= (N+\lambda )a_{0,\lambda }(N)\\
&=(N+\lambda )(N+\lambda -1)a_{0,\lambda }(N-1)\\
&=\cdots\\
&=(N+\lambda )(N+\lambda -1)\cdots(N+\lambda -(N-1))a_{0,\lambda }(1)\\
&= (N+\lambda )_{N+1},
\end{split}\end{equation}

where $(x)_N=x(x-1)\cdots(x-N+1)$, $(N \geq 1)$, and $(x)_0=1$.

Also, from \eqref{08} and \eqref{11} we have
\begin{equation}\begin{split}\label{13}
a_{N+1,\lambda }(N+1)&= (N+1)a_{N,\lambda }(N)\\
&= (N+1)Na_{N-1,\lambda }(N-1)\\
&=\cdots\\
&=(N+1)N\cdots 2 a_{1,\lambda }(1)\\
&= (N+1)!.
\end{split}\end{equation}

Let $i$ be a fixed integer with $1 \leq i \leq N$. Then, from \eqref{08} and \eqref{13}, we observe that
\begin{equation*}\begin{split}
a_{i,\lambda }(N+1)&=   (N+(i+1)\lambda )a_{i,\lambda }(N) + ia_{i-1,\lambda }(N)\\
&= (N+(i+1)\lambda ) \bigg\{ (N+ (i+1)\lambda -1) a_{i,\lambda }(N-1) + i a_{i-1,\lambda }(N-1) \bigg\}\\
&\quad +ia_{i-1,\lambda }(N)\\
&= (N+(i+1)\lambda )(N+(i+1)\lambda -1) a_{i,\lambda }(N-1)\\
&\quad + i (N+(i+1)\lambda )a_{i-1,\lambda }(N-1)+ia_{i-1,\lambda }(N)\\
&=(N+(i+1)\lambda )(N+(i+1)\lambda -1) \\
&\quad \times \bigg\{ (N+(i+1)\lambda -2)a_{i,\lambda }(N-2)+ia_{i-1,\lambda }(N-2)\bigg\}\\
&\quad + i(N+(i+1)\lambda )a_{i-1,\lambda }(N-1)+ia_{i-1,\lambda }(N)\\
&= (N+(i+1)\lambda )(N+(i+1)\lambda -1)(N+(i+1)\lambda -2) a_{i,\lambda }(N-2)\\
&\quad + i \sum_{l=0}^2 (N+(i+1)\lambda )_l a_{i-1,\lambda }(N-l) \\
&= \cdots\\
&= i! (N+(i+1)\lambda )_{N-i+1}\\
&\quad + i \sum_{l=0}^{N-i} (N+(i+1)\lambda )_l a_{i-1,\lambda }(N-l).
\end{split}\end{equation*}

Thus we have shown that
\begin{equation}\begin{split}\label{14}
a_{i,\lambda }(N+1) &=  i! (N+(i+1)\lambda )_{N-i+1}\\
&\quad + i \sum_{l=0}^{N-i} (N+(i+1)\lambda )_l a_{i-1,\lambda }(N-l),
\end{split}\end{equation}
for $1 \leq i \leq N$.

Now, from \eqref{05} and \eqref{12}-\eqref{14}, we get the following theorem.

\begin{thm}
The family of nonlinear differential equations
\begin{equation*}\begin{split}
(-1)^N (1+t)^N F^{(N)} = \sum_{i=0}^N a_{i,\lambda }(N) F^{i+1},\,\,(N=1,2,3,\cdots)
\end{split}\end{equation*}
have a solution
\begin{equation*}\begin{split}
F=F(t;\lambda )=\frac{1}{\log_\lambda (1+t)}=\frac{\lambda }{(1+t)^\lambda -1},
\end{split}\end{equation*}
where $a_{i,\lambda }(N) (0 \leq i \leq N)$ are uniquely determined by
\begin{equation}\begin{split}\label{15}
a_{0,\lambda }(N)=(N+\lambda -1)_N,
\end{split}\end{equation}

\begin{equation}\begin{split}\label{16}
a_{i,\lambda }(N)&= i! (N+(i+1)\lambda -1)_{N-i}\\
&\quad + i \sum_{l=0}^{N-i-1} (N+(i+1)\lambda -1)_l a_{i-1,\lambda }(N-l-1),\,\,(1 \leq i \leq N-1),
\end{split}\end{equation}
\begin{equation}\begin{split}\label{17}
a_{N,\lambda }(N)=N!.
\end{split}\end{equation}
\end{thm}

\section{Explicit determination of $a_{i,\lambda }(N)$ and  $b_{n,\lambda }$}

In this section, we would like to determine explicit expressions for $a_{i,\lambda }(N)$ and $b_{n,\lambda }$. For this purpose, we will make use of the exponential partial Bell polynomials $B_{n,k}(x_1,x_2,\cdots,x_{n-k+1})$ and Fa\`a di Bruno formula (see \cite{03}, p.133) which have been intensively used by F. Qi et al (see \cite{15}).

Further, we will show that our result can be expressed also in terms of the degenerate Stirling numbers of the second kind $S_{2,\lambda }(n,k)$ (see \cite{08,13}).

The exponential partial Bell polynomials $B_{n,k}(x_1,x_2,\cdots,x_{n-k+1})$ are defined by
\begin{equation}\begin{split}\label{18}
&B_{n,k}(x_1,x_2,\cdots,x_{n-k+1})\\
&= \sum \frac{n!}{\prod_{l=1}^{n-k+1}i_l!} \prod_{l=1}^{n-k+1} \left( \frac{x_l}{l!} \right)^{i_l},\,\,(n \geq k \geq 0),
\end{split}\end{equation}
where the sum runs all nonnegative integers $i_1,i_2,\cdots,i_{n-k+1}$ satisfying
$ i_1+i_2+\cdots +i_{n-k+1}=k $, and
$ i_1+2i_2+\cdots+(n-k+1)i_{n-k+1}=n$, (see \cite{03}, p.133).

The generating function for $B_{n,k} (x_1,x_2,\cdots,x_{n-k+1})$ is given by
\begin{equation}\begin{split}\label{19}
\frac{1}{k!} \left( \sum_{i=1}^\infty x_i \frac{t^i}{i!} \right)^k = \sum_{n=k}^\infty B_{n,k} (x_1,x_2,\cdots,x_{n-k+1}) \frac{t^n}{n!},\quad (\text{see}\,\, [3], \text{p}.133).
\end{split}\end{equation}

Replacing $x_i$ by $(1)_{i,\lambda }$ $(i=1,2,\cdots)$ in \eqref{19}, we easily obtain the following expressions
\begin{equation}\begin{split}\label{20}
&B_{n,k}( (1)_{1,\lambda },(1)_{2,\lambda },\cdots,(1)_{n-k+1,\lambda })\\
&=\frac{(-1)^k}{k!} \sum_{l=0}^k (-1)^l {k \choose l} (l)_{n,\lambda } = S_{2,\lambda }(n,k),\,\,(n \geq k),\quad (\text{see}\,\, (1.9),(1.10)).
\end{split}\end{equation}

Here we note that the first equality in \eqref{20} was derived in \cite{15} and the whole equalities were also obtained in (1.7) and (2.18) of \cite{07}.

The next Fa\`a di Bruno formula has played important role in many of the papers by F. Qi and his colleagues (see, for example \cite{15}).

\begin{equation}\begin{split}\label{21}
\frac{d^n}{dt^n}f \circ g(t) \\
&= \sum_{k=0}^n f^{(k)}(g(t)) B_{n,k}(g'(t),g''(t),\cdots,g^{(n-k+1)}(t) ),\quad (\text{see}\,\, [3,\,\,\text{p}.137]).
\end{split}\end{equation}

We need one more formula from (\cite{03}, p.135).
\begin{equation}\begin{split}\label{22}
&B_{n,k}(abx_1,ab^2x_2,\cdots,ab^{n-k+1}x_{n-k+1}) \\
&=a^k b^n B_{n,k}(x_1,x_2,\cdots,x_{n-k+1}),\quad (\text{see}\,\, [3,\,\text{p}.135]).
\end{split}\end{equation}

We are now going to show the following theorem about explicit expressions for $a_{i,\lambda }{(N)}$.

\begin{thm}
For $0 \leq i \leq N$,
\begin{equation}\begin{split}\label{23}
a_{i,\lambda }(N) &= (-1)^N \lambda ^{-i} \sum_{k=i}^N \sum_{l=0}^k (-1)^l {k \choose i} {k \choose l} (\lambda l)_N\\
&= (-1)^N \lambda ^{N-i} \sum_{k=i}^N (-1)^k k! {k \choose i} S_{2,\frac{1}{\lambda }}(N,k).
\end{split}\end{equation}
\end{thm}
\begin{proof}
We apply the Fa\`a di Bruno formula in \eqref{21} with
\begin{equation*}\begin{split}
u=g(t) = \frac{(1+t)^\lambda -1}{\lambda },\,\,f(u) = \frac{1}{u}.
\end{split}\end{equation*}

Then, by using \eqref{22}, we have
\begin{equation}\begin{split}\label{24}
\frac{d^N}{dt^N} F(t) &= \frac{d^N}{dt^N} f \circ g (t)\\
&= \sum_{k=0}^N \left( \frac{1}{u}\right)^{(k)} B_{n,k}(g'(t),g''(t),\cdots,g^{(N-k+1)}(t)) \\
&= \sum_{k=0}^N (-1)^k k! \frac{1}{u^{k+1}} (1+t)^{\lambda k-N} \lambda ^{N-k}\\
&\quad \times B_{N,k} \big((1)_{1,\frac{1}{\lambda }},(1)_{2,\frac{1}{\lambda }},\cdots,(1)_{N-k+1,\frac{1}{\lambda }}\big).
\end{split}\end{equation}

With $\lambda $ replaced by $\frac{1}{\lambda }$ in \eqref{20}, from \eqref{24} we obtain two different expressions.

\begin{eqnarray}
&(-1)^N (1+t)^N \frac{d^N}{dt^N}F(t) \nonumber  \\ \label{25}
&= (-1)^N  \sum_{k=0}^N (-1)^k k! \frac{1}{u^{k+1}}(1+t)^{\lambda k}\lambda^{N-k}S_{2,\frac{1}{\lambda }}(N,k)\\ \label{26}
&=(-1)^N  \sum_{k=0}^N \frac{1}{u^{k+1}}(1+t)^{\lambda k}\lambda^{-k} \sum_{l=0}^k (-1)^l {k \choose l} (\lambda l)_N.
\end{eqnarray}

Before proceeding further, we observe the following.
\begin{equation}\begin{split}\label{27}
&\frac{1}{u^{k+1}}(1+t)^{\lambda k}\lambda ^{-k}\\
&= \frac{1}{u^{k+1}} \left( \frac{(1+t)^\lambda -1}{\lambda } + \frac{1}{\lambda } \right)^k\\
&= \frac{1}{u} \big( 1+ \frac{1}{\lambda } u^{-1} \big)^k\\
&= \frac{1}{u} \sum_{i=0}^k {k \choose i} \lambda ^{-i} u^{-i}\\
&= \sum_{i=0}^k {k \choose i} \lambda ^{-i} \frac{1}{u^{i+1}}.
\end{split}\end{equation}

Substituting \eqref{27} respectively into \eqref{25} and \eqref{26} we have

\begin{eqnarray}
&(-1)^N (1+t)^N \frac{d^N}{dt^N}F(t) \nonumber  \\ \label{28}
&= (-1)^N \sum_{i=0}^N \lambda ^{N-i} \sum_{k=i}^N (-1)^k k! {k \choose i} S_{2, \frac{1}{\lambda }}(N,k) F^{i+1}\\ \label{29}
&= (-1)^N \sum_{i=0}^N  \lambda ^{-i} \sum_{k=i}^N \sum_{l=0}^k (-1)^l {k \choose i}{k \choose l} (\lambda l)_N F^{i+1}.
\end{eqnarray}

Now, by comparing \eqref{28} and \eqref{29} with \eqref{05}, the derived results follow.
\end{proof}

\begin{cor}
For $N=1,2,\cdots,$ the sequence of numbers $a_{i,\lambda }(N)$ $(0 \leq i \leq N)$ uniquely determined by the recurrence relations \eqref{15}-\eqref{17} is explicitly given by

\begin{equation*}\begin{split}
a_{i,\lambda }(N) &= (-1)^N \lambda ^{-i} \sum_{k=i}^N \sum_{l=0}^k (-1)^l {k \choose i} {k \choose l} (\lambda l)_N\\
&= (-1)^N \lambda ^{N-i} \sum_{k=i}^N (-1)^k k! {k \choose i} S_{2,\frac{1}{\lambda }}(N,k).
\end{split}\end{equation*}
\end{cor}

Next, we want to derive explicit expressions for $b_{n,\lambda }$ by using the result in the above theorem.

\begin{thm}
The degenerate Bernoulli numbers of the second kind $b_{n,\lambda }$ are given by
\begin{equation}\begin{split}\label{30}
b_{n,\lambda } &= (-1)^n \bigg\{ \frac{(1)_{n+1,\lambda }}{n+1} + \sum_{i=0}^{n-1} \frac{(1)_{i+1,\lambda }}{(i+1)!} (a_{i,\lambda }(n)-na_{i,\lambda }(n-1))\bigg\}\\
&=\frac{(-1)^n (1)_{n+1,\lambda }}{n+1} + \sum_{i=0}^{n-1} \sum_{k=i}^n \frac{(1)_{i+1,\lambda }}{(i+1)!} (-1)^k k! {k \choose i} \lambda ^{n-i-1} \\
&\quad \times (\lambda S_{2,\frac{1}{\lambda }}(n,k) + nS_{2,\frac{1}{\lambda }}(n-1,k))\\
&=\frac{(-1)^n (1)_{n+1,\lambda }}{n+1} + \sum_{i=0}^{n-1} \frac{(1)_{i+1,\lambda }}{(i+1)!}\lambda ^{-i}\\
&\quad \times \bigg\{ \sum_{l=0}^n (-1)^l {n \choose i} {n \choose l} (\lambda l)_n + \sum_{k=i}^{n-1} \sum_{l=0}^k (-1)^l {k \choose i} {k \choose l} (\lambda l+1)_n \bigg\},
\end{split}\end{equation}
where we understand that $S_{2,\frac{1}{\lambda }}(n-1,n)=0$.
\end{thm}

\begin{proof}
First, we note that
\begin{equation}\begin{split}\label{31}
b_{n,\lambda }= \lim_{t \rightarrow 0} (tF(t))^{(n)},
\end{split}\end{equation}
where, from \eqref{01} and \eqref{05}, we have
\begin{equation}\begin{split}\label{32}
&(tF(t))^{(n)} = tF(t)^{(n)}+nF(t)^{(n-1)}= \left( \frac{-1}{1+t} \right)^n\\
&\times \frac{t \sum_{i=0}^n a_{i,\lambda }(n) (\log_\lambda (1+t))^{n-i}-n(1+t)\sum_{i=0}^{n-1}a_{i,\lambda }(n-1)(\log_\lambda (1+t))^{n-i}  }{(\log_\lambda (1+t))^{n+1}}.
    \end{split}\end{equation}

Thus, from \eqref{31} and \eqref{32}, we obtain
\begin{equation}\begin{split}\label{33}
b_{n,\lambda }=(-1)^n \lim_{w \rightarrow 0 } \frac{S_{n,w,\lambda }}{w^{n+1}}.
\end{split}\end{equation}
Here
\begin{equation}\begin{split}\label{34}
S_{n,w,\lambda } &= (e_\lambda w -1)\sum_{i=0}^n a_{i,\lambda }(n) w^{n-i} -n e_\lambda  w \sum_{i=0}^{n-1} a_{i,\lambda }(n-1) w^{n-i}\\
&= (e_\lambda w -1) a_{n,\lambda }(n) - \sum_{i=0}^{n-1} a_{i,\lambda }(n) w^{n-i}\\
&\quad + \sum_{i=0}^{n-1} (a_{i,\lambda }(n) - n a_{i,\lambda }(n-1) ) w^{n-i} e_\lambda w\\
&= n! \sum_{j=1}^\infty \frac{(1)_{j,\lambda }}{j!} w^j - \sum_{j=1}^n a_{n-j,\lambda }(n) w^{j}\\&\quad +
\sum_{i=1}^n (a_{n-i,\lambda }(n) - n a_{n-i,\lambda }(n-1))w^i \sum_{l=0}^\infty \frac{(1)_{l,\lambda }}{l!}w^l\\
&= n! \sum_{j=1}^\infty \frac{(1)_{j,\lambda }}{j!} w^j -\sum_{j=1}^n a_{n-j,\lambda }(n) w^{j}\\&\quad +
\sum_{j=1}^\infty \sum_{i=1}^{\min\left\{n,j\right\}}  \frac{(1)_{j-i,\lambda }}{(j-i)!} (a_{n-i,\lambda }(n) -n a_{n-i,\lambda }(n-1)) w^j\\
&=\sum_{j=1}^n  \bigg\{ n! \frac{(1)_{j,\lambda }}{j!}- a_{n-j,\lambda }(n) + \sum_{i=1}^j \frac{(1)_{j-i,\lambda }}{(j-i)!} (a_{n-i,\lambda }(n) - na_{n-i,\lambda }(n-1)) \bigg\} w^j\\
&+ \sum_{j=n+1}^\infty \bigg\{ n! \frac{(1)_{j,\lambda }}{j!} + \sum_{i=1}^n \frac{(1)_{j-i,\lambda }}{(j-i)!} (a_{n-i,\lambda }(n) - na_{n-i,\lambda }(n-1)) \bigg\} w^j.\\
\end{split}\end{equation}

Hence, from \eqref{33} and \eqref{34}, we see that
\begin{equation}\begin{split}\label{35}
&\sum_{i=1}^j \frac{(1)_{j-i,\lambda }}{(j-i)!} \big(a_{n-i,\lambda }(n) - na_{n-i,\lambda }(n-1)\big)\\
&=a_{n-j,\lambda }(n) - n! \frac{(1)_{j,\lambda }}{j!},\,\,(1 \leq j \leq n),
\end{split}\end{equation}
and
\begin{equation}\begin{split}\label{36}
b_{n,\lambda } &= (-1)^n \bigg\{ n! \frac{(1)_{n+1,\lambda }}{(n+1)!} + \sum_{i=1}^{n} \frac{(1)_{n+1-i,\lambda }}{(n+1-i)!} (a_{n-i,\lambda }(n)-na_{n-i,\lambda }(n-1))\bigg\}\\
&= (-1)^n \bigg\{ \frac{(1)_{n+1,\lambda }}{n+1} + \sum_{i=0}^{n-1} \frac{(1)_{i+1,\lambda }}{(i+1)!} (a_{i,\lambda }(n)-na_{i,\lambda }(n-1))\bigg\}\\
\end{split}\end{equation}

Finally, the rest of assertions in \eqref{30} follows from \eqref{36} together with \eqref{23}.
\end{proof}

During the course of the proof of the above theorem, we have shown the following result.

\begin{cor}
For each $n=1,2,\cdots,$ the sequences $\{ a_{j,\lambda }(n)\}_{j=0}^{n-1}$ and $\{ a_{j,\lambda }(n-1)\}_{j=0}^{n-1}$ satisfy the following identities.
\begin{equation*}\begin{split}
&\sum_{i=1}^j \frac{(1)_{j-i,\lambda }}{(j-i)!} \big( a_{n-i,\lambda }(n) - na_{n-i,\lambda }(n-1) \big)\\
&= a_{n-j,\lambda }(n) - n! \frac{(1)_{j,\lambda }}{j!},\,\,(1 \leq j \leq n).
\end{split}\end{equation*}
\end{cor}

\begin{thm} For $N \geq k \geq 0$, we have
\begin{equation}\begin{split}\label{37}
&\lim_{\lambda \rightarrow 0} \lambda ^{N-k} S_{2,\frac{1}{\lambda }} (N,k) \\
&= \lim_{\lambda  \rightarrow 0 } \lambda ^{N-k} B_{N,k} \big( (1)_{1,\frac{1}{\lambda }},(1)_{2,\frac{1}{\lambda }},\cdots,(1)_{N-k+1,\frac{1}{\lambda }} \big) = S_1(N,k),
\end{split}\end{equation}
where $S_1(N,k)$ are the signed Stirling numbers of the first kind.
\end{thm}

\begin{proof}
Here we give two different proofs for this. From \eqref{20} and \eqref{22}, we note that
\begin{equation}\begin{split}\label{38}
&B_{N,k}\big(1,(\lambda -1),(\lambda -1)(\lambda -2),\cdots,(\lambda -1),(\lambda -2)\cdots(\lambda -(N-k)) \big)\\
&= \lambda ^{N-k} B_{N,k} \Big( (1)_{1,\frac{1}{\lambda }}, (1)_{2,\frac{1}{\lambda }}, \cdots, (1)_{N-k+1,\frac{1}{\lambda }}\Big) \\
&= \lambda ^{N-k} S_{2, \frac{1}{\lambda }}(N,k).
\end{split}\end{equation}
Thus from \eqref{38}, we have
\begin{equation*}\begin{split}
&\lim_{\lambda \rightarrow 0} \lambda ^{N-k} S_{2,\frac{1}{\lambda }}(N,k)\\
&=B_{N,k}\Big( (-1)^0 0!, (-1)^1 1!, (-1)^2 2!, \cdots, (-1)^{N-k}(N-k)! \Big)\\
&= (-1)^{N-k} B_{N,k} \Big(0!, 1!, 2!, \cdots, (N-k)! \Big)\\
&= (-1)^{N-k} | S_1(N,k) |\\
&= S_1(N,k),\quad (\text{see}\,\, [3,\,\text{p}.135]).
\end{split}\end{equation*}
For another proof, we consider
\begin{equation}\begin{split}\label{39}
&\sum_{N=k}^\infty \lambda ^{N-k} B_{N,k}\Big( (1)_{1,\frac{1}{\lambda }},(1)_{2, \frac{1}{\lambda }},\cdots,(1)_{N-k+1,\frac{1}{\lambda }} \Big) \frac{t^N}{N!}\\
&= \lambda ^{-k} \sum_{N=k}^\infty  B_{N,k}\Big( (1)_{1,\frac{1}{\lambda }},(1)_{2, \frac{1}{\lambda }},\cdots,(1)_{N-k+1,\frac{1}{\lambda }} \Big) \frac{(\lambda t)^N}{N!}\\
&= \lambda ^{-k} \frac{1}{k!} \left( \sum_{i=1}^\infty (1)_{i,\frac{1}{\lambda }} \frac{(\lambda t)^i}{i!} \right)^k\\
&= \lambda ^{-k} \frac{1}{k!} \Big( (1+t)^\lambda -1\big)^k\\
&= \frac{1}{k!} \Big(\log_\lambda (1+t)\Big)^k.
\end{split}\end{equation}
The statement follows now from \eqref{38} and \eqref{39}.
\end{proof}

\begin{cor}
For $0 \leq i \leq N$,
\begin{equation}\begin{split}\label{40}
\lim_{\lambda \rightarrow 0} a_{i,\lambda }(N) = (-1)^N (-1)^i i! S_1(N,i).
\end{split}\end{equation}
\end{cor}

\begin{proof}
From \eqref{23}, we have
\begin{equation*}\begin{split}
a_{i,\lambda }(N) = (-1)^N \sum_{k=i}^N \lambda ^{k-i}(-1)^k k! {k \choose i}\lambda^{N-k} S_{2,\frac{1}{\lambda }}(N,k).
\end{split}\end{equation*}

Now, the result follows from \eqref{37}.
\end{proof}

Making use of \eqref{37} and taking $\lambda \rightarrow 0$ in \eqref{24}, we get
\begin{equation}\begin{split}\label{41}
\left( \frac{1}{\log(1+t)} \right)^{(N)} = \frac{1}{(1+t)^N} \sum_{k=0}^N (-1)^k k! \frac{S_1(N,k)}{(\log(1+t))^{k+1}},\,\,(N \geq 0),
\end{split}\end{equation}
which agrees with the results in \cite{14, 15}.

In the same way, invoking \eqref{40} and taking $\lambda \rightarrow 0$ in \eqref{32}, we obtain
\begin{equation}\begin{split}\label{42}
&\left( \frac{t}{\log(1+t)} \right)^{(n)}\\
&= \frac{1}{(1+t)^n } \sum_{i=0}^n \frac{(-1)^i i! (t S_1(n,i) + n(1+t)S_1(n-1,i) )}{(\log(1+t))^{i+1}},\,\,(n \geq 1),
\end{split}\end{equation}
which again agrees with the results in \cite{14, 15}.

Here we understand that $S_1(n-1,n) =0$.

The Bernoulli numbers of the second kind $b_n$ are defined by the generating function
\begin{equation}\begin{split}\label{43}
\frac{t}{\log(1+t)} = \sum_{n=0}^\infty       b_n \frac{t^n}{n!}.
\end{split}\end{equation}

Thus from \eqref{107} we see that $b_n = \lim_{\lambda  \rightarrow 0}b_{n,\lambda }$. By taking the limit $\lambda  \rightarrow 0$ of the expression in \eqref{30} and making use of either \eqref{37} or \eqref{40}, we obtain
\begin{equation}\begin{split}\label{44}
b_n = \sum_{i=0}^n \frac{(-1)^i}{i+1} \Big( S_1(n,i) + n S_1(n-1,i) \Big),\,\,(n \geq 1),
\end{split}\end{equation}
where $S_1(n-1,n) =0$. Again, we see that \eqref{44} agrees with the result in (\cite{14}, (1.4), (2.18), (3.11)).

Next, we will present two different expansions for $b_{n,\lambda }.$

We first observe that
\begin{equation}\begin{split}\label{45}
\frac{(1+t)^\lambda -1}{\lambda t} = \sum_{n=0}^\infty      \frac{(\lambda -1)_n}{n+1}  \frac{t^n}{n!}.
\end{split}\end{equation}

Thus from \eqref{45} we get
\begin{equation}\begin{split}\label{46}
1&= \left( \sum_{l=0}^\infty b_{l,\lambda } \frac{t^l}{l!} \right) \left( \sum_{m=0}^\infty \frac{(\lambda -1)_m}{m+1} \frac{t^m}{m!} \right) \\
&= \sum_{n=0}^\infty   \left( \sum_{l=0}^n \frac{1}{n-l+1} {n \choose l} (\lambda -1)_{n-l} b_{l,\lambda } \right)    \frac{t^n}{n!}.
\end{split}\end{equation}

Now, we derive the following recurrence relation for $b_{n,\lambda }$ from \eqref{46}:
\begin{equation}\begin{split}\label{47}
b_{0,\lambda }=1,\,\, b_{n,\lambda } = - \sum_{l=0}^{n-1} \frac{1}{n-l+1} {n \choose l} (\lambda -1)_{n-l} b_{l,\lambda },\,\,(n \geq 1).
\end{split}\end{equation}

In order to find another expression for $b_{n,\lambda }$, we consider the following.
\begin{equation}\begin{split}\label{48}
&\sum_{n=0}^\infty       b_{n,\lambda } \frac{t^n}{n!} = \frac{1}{1- (1- \frac{(1+t)^\lambda -1}{\lambda t})}\\
&=\sum_{k=0}^\infty (-1)^k \left( \frac{(1+t)^\lambda -1}{\lambda t} -1 \right)^k\\
&= \sum_{k=0}^\infty (-1)^k \left( \sum_{m=1}^\infty  \frac{(\lambda -1)_m}{m+1}\frac{t^m}{m!} \right)^k\\
&= \sum_{k=0}^\infty (-1)^k \sum_{n=k}^\infty \left( \sum_{\substack{m_1+\cdots+m_k=n\\m_i \geq 1}} {n \choose m_1,\cdots,m_k } \frac{(\lambda -1)_{m_1}}{m_1+1}\cdots\frac{(\lambda -1)_{m_k}}{m_k+1} \right) \frac{t^n}{n!}\\
&= \sum_{n=0}^\infty \left( \sum_{k=0}^n (-1)^k \sum_{\substack{m_1+\cdots+m_k=n\\m_i \geq 1}}  {n \choose m_1,\cdots,m_k }    \frac{(\lambda -1)_{m_1}}{m_1+1}\cdots\frac{(\lambda -1)_{m_k}}{m_k+1}. \right)   \frac{t^n}{n!}.
\end{split}\end{equation}

Thus from \eqref{48}  we see that
\begin{equation}\begin{split}\label{49}
b_{n,\lambda } = \sum_{k=0}^n (-1)^k \sum_{\substack{m_1+\cdots+m_k=n\\m_i \geq 1}}  {n \choose m_1,\cdots,m_k }    \frac{(\lambda -1)_{m_1}}{m_1+1}\cdots\frac{(\lambda -1)_{m_k}}{m_k+1}.
\end{split}\end{equation}

\section{Applications of differential equations}

For any positive integer $r$, the degenerate Bernoulli numbers of the second kind of order $r$ are defined by
\begin{equation}\begin{split}\label{50}
\left( \frac{t}{\log_\lambda (1+t)} \right)^r = \left(\frac{\lambda t}{(1+t)^\lambda -1}
\right)^r = \sum_{n=0}^\infty   b_{n,\lambda }^{(r)}      \frac{t^n}{n!}.
\end{split}\end{equation}

Here, using the nonlinear differential equation in Theorem 2.1, we will derive an identity which expresses the degenerate Bernoulli numbers of the second kind in terms of those numbers of higher-order.

We recall here that the degenerate Bernoulli numbers of the second kind $b_{n,\lambda }$ are given by
\begin{equation}\begin{split}\label{51}
\frac{t}{\log_\lambda (1+t)} = tF(t) = \sum_{j=0}^\infty b_{j,\lambda } \frac{t^j}{j!}.
\end{split}\end{equation}

On the one hand, from \eqref{51} we have
\begin{equation}\begin{split}\label{52}
\sum_{j=0}^\infty b_{j+N,\lambda } \frac{t^j}{j!} = \Big( tF(t) \Big)^{(N)}.
\end{split}\end{equation}

On the other hand, for $N \geq 1$ and from Theorem 2.1 we have
\begin{equation}\begin{split}\label{53}
&\Big(tF(t)\Big)^{(N)} = t F(t) ^{(N)} + N F(t) ^{(N-1)}\\
&=\left( \frac{-1}{1+t} \right)^N \bigg\{ t \sum_{i=0}^N a_{i,\lambda }(N) F^{i+1} - N(1+t) \sum_{i=0}^{N-1} a_{i,\lambda }(N-1) F^{i+1} \bigg\}\\
&=\left( \frac{-1}{1+t} \right)^N \bigg\{\sum_{i=0}^N a_{i,\lambda }(N) t^{-i} \Big( tF\Big)^{i+1} - N \sum_{i=0}^{N-1} a_{i,\lambda }(N-1) t^{-i-1} \Big(tF\Big)^{i+1} \\
&\quad -N \sum_{i=0}^{N-1} a_{i,\lambda }(N-1) t^{-i} \Big( tF \Big)^{i+1} \bigg\}\\
&= \sum_{m=0}^\infty (-1)^{N+m} {N+m-1 \choose m} t^m \\
&\quad \times \bigg\{ \sum_{i=0}^N \sum_{l=0}^\infty a_{i,\lambda }(N) b_{l,\lambda }^{(i+1)} \frac{t^{l-i}}{l!} -N \sum_{i=0}^{N-1} \sum_{l=0}^\infty a_{i,\lambda }(N-1) b_{l,\lambda }^{(i+1)} \frac{t^{l-i-1}}{l!}\\
&\qquad - N \sum_{i=0}^{N-1} \sum_{l=0}^\infty a_{i,\lambda }(N-1) b_{l,\lambda }^{(i+1)} \frac{t^{l-i}}{l!} \bigg\}\\
&=\sum_{m=0}^\infty \sum_{i=0}^N \sum_{l=0}^\infty (-1)^{N+m} {N+m-1 \choose m} a_{i,\lambda }(N) b_{l,\lambda }^{(i+1)} \frac{t^{m+l-i}}{l!} \\
&\quad -N\sum_{m=0}^\infty \sum_{i=0}^{N-1} \sum_{l=0}^\infty (-1)^{N+m} {N +m -1 \choose m} a_{i,\lambda }(N-1) b_{l,\lambda }^{(i+1)}  \frac{t^{m+l-i-1}}{l!} \\
&\quad -N \sum_{m=0}^\infty \sum_{i=0}^{N-1} \sum_{l=0}^\infty (-1)^{N+m} {N+m-1 \choose m} a_{i,\lambda }(N-1) b_{l,\lambda }^{(i+1)} \frac{t^{m+l-i}}{l!}.
\end{split}\end{equation}

The first sum in \eqref{53} is
\begin{equation}\begin{split}\label{54}
\sum_{m=0}^\infty \sum_{i=0}^N \sum_{l=0}^\infty C_{m,i,l}(N) t^{m+l-i},
\end{split}\end{equation}

with
\begin{equation*}\begin{split}
C_{m,i,l} (N) = (-1)^{N+m} {N+m-1 \choose m} a_{i,\lambda }(N) b_{l,\lambda }^{(i+1)} \frac{1}{l!}.
\end{split}\end{equation*}

We observe that \eqref{54} is equal to
\begin{equation}\begin{split}\label{55}
&\sum_{m=0}^\infty \sum_{i=0}^N \sum_{l=-i}^\infty  C_{m,i,l+i}(N) t^{m+l} \\
&= \sum_{m=0}^\infty \sum_{l=0}^\infty \sum_{i=0}^N C_{m,i,l+i} (N) t^{m+l}
+ \sum_{m=0}^\infty \sum_{l=-N}^{-1} \sum_{i=-l}^N C_{m,i,l+i} (N) t^{m+l} \\
&= \sum_{j=0}^\infty \sum_{l=0}^j \sum_{i=0}^N C_{j-l, i,l+i}(N) t^j
+ \sum_{j=-N}^\infty \sum_{l=-N}^{\min\left\{-1,j\right\}} \sum_{i=-l}^N C_{j-l, i,l+i}(N) t^j \\
&= \sum_{j=0}^\infty \sum_{l=0}^j \sum_{i=0}^N C_{j-l, i, l+i}(N) t^j
+ \sum_{j=0}^\infty \sum_{l=-N}^{-1} \sum_{i=-l}^N C_{j-l,i,l+i} (N) t^j \\
& \quad + \sum_{j=-N}^{-1} \sum_{l=-N}^j \sum_{i=-l}^N C_{j-l,i,l+i}(N) t^j\\
&= \sum_{j=0}^\infty \sum_{l=-N}^j \sum_{i=\max\left\{0,-l\right\}}^N C_{j-l,i,l+i}(N) t^j
+ \sum_{j=-N}^{-1} \sum_{l=-N}^j \sum_{i=\max\left\{0,-l\right\}}^N C_{j-l,i,l+i}(N) t^j\\
&= \sum_{j=-N}^\infty  \sum_{l=-N}^j \sum_{i=\max\left\{0,-l\right\}}^N C_{j-l,i,l+i}(N) t^j
\end{split}\end{equation}
\begin{equation*}\begin{split}
&= \sum_{j=-N}^\infty \sum_{l=-N}^j \sum_{i=\max\left\{0,-l\right\}}^N (-1)^{N+j+l} {N+j-l-1 \choose j-l} a_{i,\lambda }(N) b_{l+i,\lambda }^{(i+1)} \frac{1}{(l+i)!} t^j.
\end{split}\end{equation*}

Proceeding in the same way, the second and third sums in \eqref{53} are respectively given by
\begin{equation}\begin{split}\label{56}
&\sum_{j=-N}^\infty \sum_{l=-N}^j \sum_{i=\max\left\{0,-l-1\right\}}^{N-1} N(-1)^{N+j+l+1} {N+j-l-1 \choose j-l}\\
&\times a_{i,\lambda } (N-1) b_{l+i+1,\lambda }^{(i+1)} \frac{1}{(l+i+1)!} t^j,
\end{split}\end{equation}
\begin{equation}\begin{split}\label{57}
&\sum_{j=-(N-1)}^\infty \sum_{l=-(N-1)}^j \sum_{i=\max\left\{0,-l\right\}}^{N-1} N (-1)^{N+j+l+1} {N+j-l-1 \choose j-l}\\
&\times a_{i,\lambda }(N-1) b_{l+i,\lambda }^{(i+1)} \frac{1}{(l+i)!} t^j.
\end{split}\end{equation}

Now, from \eqref{55}-\eqref{57} and separating the sum into the singular and nonsingular ones we obtain

\begin{equation}\begin{split}\label{58}
&\Big(t F(t) \Big)^{(N)}\\
&= \sum_{j=0}^\infty \bigg\{ \sum_{l=-N}^j \sum_{i=\max\left\{0,-l\right\}}^N (-1)^{N+j+l} {N+j-l-1 \choose j-l} a_{i,\lambda }(N) b_{l+i,\lambda }^{(i+1)} \frac{j!}{(l+i)!} \\
&\quad + \sum_{l=-N}^j \sum_{i=\max\left\{0,-l-1\right\}}^{N-1}  N(-1)^{N+j+l+1} {N+j-l-1 \choose j-l}\\
&\qquad \times  a_{i,\lambda }(N-1) b_{l+i+1,\lambda }^{(i+1)} \frac{j!}{(l+i+1)!}\\
&\quad +\sum_{l=-(N-1)}^j \sum_{i=\max\left\{0,-l\right\}}^{N-1} N (-1)^{N+j+l+1} {N+j-l-1 \choose j-l}\\&\qquad \times  a_{i,\lambda }(N-1)  b_{l+i,\lambda }^{(i+1)} \frac{j!}{(l+i)!} \Bigg\} \frac{t^j}{j!}\\
&\quad + \sum_{j=-(N-1)}^{-1} \bigg\{ \sum_{l=-N}^j \sum_{i=-l}^N (-1)^{N+j+l} {N+j-l-1 \choose j-l}\\& \qquad \times a_{i,\lambda }(N) b_{l+i,\lambda }^{(i+1)} \frac{1}{(l+i)!}\\
&= \sum_{l=-N}^j \sum_{i=-l-1}^{N-1} N (-1)^{N+j+l+1} {N+j-l-1 \choose j-l} \\
&\qquad \times a_{i,\lambda }(N-1) b_{l+i+1,\lambda}^{(i+1)}  \frac{1}{(l+i+1)!}\\
&\quad + \sum_{l=-(N-1)}^j \sum_{i=-l}^{N-1} N (-1)^{N+j+l+1} {N+j-l-1 \choose j-l} \\
&\qquad a_{i,\lambda }(N-1) b_{l+i,\lambda }^{(i+1)} \frac{1}{(l+i)!} \bigg\} t^j.
\end{split}\end{equation}

Finally, comparing \eqref{52} and \eqref{58} we get the following theorem and corollary.

\begin{thm}
For $j=0,1,2,\cdots,$ and $N=1,2,3,\cdots,$ we have the following identity.
\begin{equation*}\begin{split}
&b_{j+N,\lambda } \\
&= \sum_{l=-N}^j \sum_{i=\max\left\{0,-l\right\}}^N (-1)^{N+j+l} {N +j-l-1 \choose j-l}\\
&\qquad \times  a_{i,\lambda }(N) b_{l+i,\lambda }^{(i+1)} \frac{j!}{(l+i)!}\\
&\quad + \sum_{l=-N}^j \sum_{i=\max\left\{0,-l-1\right\}}^{N-1} N (-1)^{N+j+l+1} {N+j-l-1 \choose j-l} \\
&\qquad a_{i,\lambda }(N-1) b_{l+i+1,\lambda }^{(i+1)} \frac{j!}{(l+i+1)!}\\
&\quad + \sum_{l=-(N-1)}^j \sum_{i=\max\left\{0,-l\right\}}^{N-1} N (-1)^{N+j+l+1} {N+j-l-1 \choose j-l} \\ &\qquad \times a_{i,\lambda } (N-1) b_{l+i,\lambda }^{(i+1)} \frac{j!}{(l+1)!}.
\end{split}\end{equation*}
\end{thm}

\begin{cor}
For $N \geq 2$ and $-(N-1) \leq j \leq -1$, the following identity holds true.
\begin{equation*}\begin{split}
&\sum_{l=-N}^j \sum_{i=-l}^N (-1)^{N+j+l} {N+j-l-1 \choose j-l} a_{i,\lambda }(N) b_{l+i,\lambda }^{(i+1)} \frac{1}{(l+i)!}\\
&+ \sum_{l=-N}^j \sum_{i=-l-1}^{N-1} N (-1)^{N+j+l+1} {N+j-l-1 \choose j-l} a_{i,\lambda }(N-1) b_{l+i+1,\lambda}^{(i+1)}  \frac{1}{(l+i+1)!}\\
&+ \sum_{l=-(N-1)}^j \sum_{i=-l}^{N-1} N (-1)^{N+j+l+1} {N+j-l-1 \choose j-l} a_{i,\lambda }(N-1) b_{l+i,\lambda }^{(i+1)} \frac{1}{(l+i)!} = 0.
\end{split}\end{equation*}

\end{cor}

\end{document}